\documentclass[11pt]{amsart}
\usepackage{amsmath,amssymb,mathrsfs,color}

\newtheorem{thm}{Theorem}[section]
\newtheorem{cor}[thm]{Corollary}
\newtheorem{lem}[thm]{Lemma}
\newtheorem{prop}[thm]{Proposition}

\theoremstyle{definition}
\newtheorem{de}[thm]{Definition}
\theoremstyle{remark}
\newtheorem{rem}[thm]{Remark}
\newtheorem{exam}[thm]{Example}
\numberwithin{equation}{section}

\allowdisplaybreaks

\begin{document}

\title[Period solutions for MVSDEs]
{Periodic solutions for McKean-Vlasov SDEs under periodic distribution-dependent Lyapunov conditions}

\author{Jun Ma}
\address{J. Ma: School of Mathematics and Statistics and Center for Mathematics and Interdisciplinary Sciences,
Northeast Normal University, Changchun, 130024, P. R. China}
\email{mathmajun@163.com}



\date{}



\begin{abstract}
In this paper, we prove the existence of periodic solutions
for McKean-Vlasov SDEs under periodic distribution-dependent Lyapunov conditions,
which is obtained by periodic Markov processes with state space $\mathbb R^d\times \mathcal P(\mathbb R^d)$.
Here $\mathcal P(\mathbb R^d)$ denotes the space of probability measures on $\mathbb R^d$.
In addition, we show the convergence to the periodic solution
and the continuous dependence on parameters of periodic solutions for McKean-Vlasov SDEs.
Finally, we provide several examples to illustrate our theoretical results.
\end{abstract}

\maketitle

\section{Introduction}
Periodic motion is a fundamental concept in dynamical systems and probability theory,
characterized by the repetitive return of a system to its original position after a fixed period of time.
Investigating periodic orbits could provide valuable insights into the long-term behavior of dynamical systems and Markov processes.
Periodic orbits have been widely observed in various natural phenomena,
such as the swinging of pendulum, the movement of sun and moon, and the periodic oscillation of electromagnetic waves.

At present, periodic orbits have received widespread attention,
and been extensively studied in dynamical systems.
Let us review some studies on periodic solutions for SDEs which are closely related to our work.
Khasminskii \cite{Khasminskii} defined periodic solutions for SDEs in the framework of periodic Markov processes,
and obtained the existence under periodic Lyapunov conditions.
Chen et al. \cite{CHLY} studied the existence of classical periodic solutions for the corresponding Fokker-Planck equations
under uncommon Lyapunov conditions.
Cheban and Liu \cite{CL} investigated the existence of periodic solutions for semilinear SDEs,
who also showed the existence of other recurrence solutions, such as quasi-periodic, almost periodic, almost automorphic and Birkhoff recurrent.
Ji et al. \cite{JQSY1,JQSY2,JQSY3} obtained the existence, uniqueness, noise-vanishing concentration and limit behaviors
of periodic solutions for the corresponding Fokker-Planck equations under periodic Lyapunov conditions.
Liu and Lu \cite{LL} studied the ergodic property of inhomogeneous Markov processes,
and obtained exponential ergodicity under total variation distance for inhomogeneous SDEs
with almost periodic coefficients under Lyapunov conditions.

In this paper, we study the existence, convergence and continuous dependence on parameters of periodic solutions
for McKean-Vlasov SDEs (abbreviated as MVSDEs)
under periodic distribution-dependent Lyapunov conditions,
where the periodic Lyapunov functions are defined on $\mathbb R^+\times\mathbb R^d\times \mathcal P(\mathbb R^d)$,
i.e. the Lyapunov functions are periodic in time, and depend not only on space variable but also on distribution variable.
MVSDEs are also known as mean-field SDEs or distribution-dependent SDEs,
whose coefficients depend not only on the microcosmic site,
but also on the macrocosmic distribution of particles
\begin{equation}\label{main}
dX_t= b(t,X_t,\mathcal{L}_{X_t})dt+ \sigma(t,X_t,\mathcal{L}_{X_t})dW_t,
\end{equation}
where $\mathcal{L}_{X_t}$ denotes the law of $X_t$.
MVSDEs are a model for Vlasov equation and were first studied by McKean \cite{Mckean} in 1966.
The study of Lasry and Lions \cite{Larsy-Lions1,Larsy-Lions2,Larsy-Lions3} as well as Huang, Malham\'e and Caines \cite{HMC1,HMC2}
has significantly attracted the interest of more researchers to MVSDEs,
who independently introduced mean-field games in order to study large population deterministic and stochastic differential games.
Wang \cite{Wang_18}, Ding and Qiao \cite{DQ} as welll as Liu and Ma \cite{LM2}
obtained the existence and uniqueness of solutions for MVSDEs.
Ren et al. \cite{RTW} and Hong et al. \cite{HHL} investigated the existence and uniqueness of solutions
for McKean-Vlasov stochastic partial differential equations.
Butkovsky \cite{Butkovsky}, Wang \cite{Wang_18}, Bogachev et al. \cite{BRS} as well as Liu and Ma \cite{LM2}
studied the existence of invariant measures for MVSDEs.
Buckdahn et al. \cite{BLP,BLPR} investigated the relationship between the functional of the form $Ef(t, \bar X_t, \mathcal L_{X_t})$ and the associated second-order PDE.

Regarding the existence of periodic solutions for MVSDEs,
Zhou et al. \cite{ZXJL} established the results by comparison principle under Lipschitz conditions.
Ren et al. \cite{RSW} obtained the exponential ergodicity for time-periodic MVSDEs under monotone conditions,
thus get the existence of periodic solutions for MVSDEs.
Sun and Wong \cite{SW} studied the existence of periodic solutions for MVSDEs under Lypapunov conditions,
whose Lyapunov functions only depend on time and space variable, not depend on distribution variable.
However, in this paper we aim to determine whether a periodic solution exists for MVSDEs under periodic distribution-dependent Lyapunov conditions.

In this paper, we focus on studying the existence and convergence of periodic solutions under periodic Lyapunov conditions,
where the periodic Lyapunov functions are defined on $\mathbb R^+\times\mathbb R^d\times \mathcal P_2(\mathbb R^d)$.
Specifically, the Lyapunov functions are periodic in time, and depend not only on space variable but also on distribution variable.
Here $\mathcal P_2(\mathbb R^d)$ denotes the set of probability measures on $\mathbb R^d$ with finite second-moment.
Considering distribution-dependent Lyapunov functions is both reasonable and natural
since the coefficients of MVSDEs depend on distribution variable.
However, the existing study under Lyapunov conditions for MVSDEs focuses on space-dependent Lyapunov functions,
which are independent of distribution variable, such as \cite{BRS,RTW}.
In \cite{LM2}, Liu and Ma simultaneously truncate space and distribution variables, as proposed in \cite{RTW},
to establish the existence and uniqueness of solutions, and employ Krylov-Bogolioubov theorem to get ergodicity under distribution-dependent Lyapunov conditions.
It is worth noting that Krylov-Bogolioubov theorem is valid in \cite {LM2} since the Lyapunov functions are defined on $\mathbb R^d\times \mathcal P_2(\mathbb R^d)$,
If the Lyapunov functions are only defined on $\mathbb R^d$, Krylov-Bogolioubov theorem may be not valid
since the corresponding Fokker-Planck equation for MVSDE is nonlinear.
Therefore, we build upon the previous idea \cite{LM2} to investigate the existence and convergence of periodic solutions
for MVSDEs under periodic distribution-dependent Lyapunov conditions.

To establish the existence of periodic solutions for MVSDEs under periodic distribution-dependent Lyapunov conditions,
we adopt the Khasminskii's approach \cite{Khasminskii},
which established the existence of periodic solutions for SDEs under periodic Lyapunov conditions by periodic Markov processes.
Following Khasminskii's idea,
we first define periodic Markov process with state space $\mathbb R^d\times \mathcal P(\mathbb R^d)$
and its corresponding transition probability function,
then derive the equivalent conditions for the existence of the periodic Markov process.
Then utilizing the equivalent conditions, we get the existence of periodic probability
on $\mathbb R^d\times \mathcal P(\mathbb R^d)$ for the transition probability function generated by the coupled MVSDE.
This coupled MVSDE is introduced since the second variable `$(x,\mu)$'
in transition probability function $P(\cdot,(x,\mu),\cdot,\cdot)$ is not a `pair',
that is to say that $\mu$ is arbitrary chosen from $\mathcal P(\mathbb R^d)$, and may not equal $\delta_x$.
Consequently, to use the results about the existence periodic solution for transition probability, we need to introduce the coupled MVSDEs;
for details, see Section 4.
Finally we obtain the existence of periodic solutions for MVSDEs under periodic distribution-dependent Lyapunov conditions,
where the periodic distribution is the ``projection" of periodic probability on $\mathbb R^d\times \mathcal P(\mathbb R^d)$
for a special coupled MVSDE;
for details, see Section 4.

In addition to establishing the existence of periodic solutions for MVSDEs,
we investigate the convergence of periodic solutions
for MVSDEs under periodic distribution-dependent Lyapunov conditions in weak topology sense, specifically using L\'evy-Prohorov distance.
We first get the convergence of periodic Markov processes,
which is exactly the convergence for some subsequence.
We then obtain the convergence of periodic solutions for MVSDEs under periodic Lyapunov conditions,
which could also be viewed as a special case of the coupled MVSDEs.
Furthermore, we study the continuous dependence on parameters of periodic solutions
for MVSDEs under distribution-dependent Lyapunov conditions.

The rest of this paper is arranged as follows.
In Section 2, we collect a number of preliminary,
including the definition of Lions derivative, function spaces and so on.
Section 3 defines periodic Markov process with state space $\mathbb R^d\times\mathcal P(\mathbb R^d)$
and its transition probability function,
and obtains the existence, convergence as well as continuous dependence on parameters of the periodic Markov processes.
In Section 4, we show the existence, convergence and continuous dependence on parameters of periodic solutions for MVSDEs
under periodic distribution-dependent Lyapunov conditions.
In Section 5, we provide some examples to illustrate our theoretical results.

\section{Preliminary}

Throughout the paper, let $(\Omega, \mathcal F, \{\mathcal F_t\}_{t\ge 0},\mathbb P)$ be a filtered complete probability space,
in which the filtration $\{\mathcal F_t\}$ is assumed to satisfy the usual condition,
i.e. it is right continuous and $\mathcal F_0$ contains all $P$-null sets.
Let $\mathcal P(\mathbb R^d)$ be the space of probability measures on $\mathbb R^d$.
We use the Wasserstein distance $W_2$ on $\mathcal P_2(\mathbb R^d)$ in what follows, i.e.
\begin{equation}
W_2(\mu, \nu):=\inf_{\pi \in \mathcal C(\mu, \nu)} \bigg[\int_{\mathbb R^d\times\mathbb R^d} |x-y|^2 \pi(dx, dy)\bigg]^{1/2}\nonumber
\end{equation}
for $\mu, \nu\in \mathcal P_2(\mathbb R^d)$,
where $\mathcal P_2(\mathbb R^d):=\{\mu\in\mathcal P(\mathbb R^d): \int_{\mathbb R^d}|x|^2 \mu(dx)<\infty \}$,
and $\mathcal C(\mu,\nu)$ denotes the set of all couplings between $\mu$ and $\nu$.

We also define an another distance $\omega:\mathcal P(\mathbb R^d)\times\mathcal P(\mathbb R^d)\to\mathbb R^+$,
which is called L\'evy-Prohorov distance, by
$$
\omega(\mu,\nu)=\inf\{\delta:\mu(A)<\nu(A^{\delta})+\delta,\nu(A)<\mu(A^{\delta})+\delta \quad \hbox{\text{for} }\hbox{\text{all} } \hbox{\text{closed} } \hbox{\text{set} } A\in\mathcal F\},
$$
where $A^{\delta}=\{x: \exists y\in A ~\hbox{\text{such} } \hbox{\text{that} }~|x-y|<\delta\}$.
Convergence under the L\'evy-Prohorov distance is equivalent to the weak convergence of measures; see \cite[Theorem 1.11]{Prohorov} for instance.
We next present a proposition that characterizes the relationship between L\'evy-Prohorov distance and Wasserstein distance
on $\mathcal P(\mathbb R^d)$ as follows; see \cite[Lemma 5.3, Theorem 5.5]{Chen} or \cite[Theorem 6.9]{Villani} for instance.
\begin{prop}\label{Ww}
For any $\mu,\nu\in\mathcal P(\mathbb R^d)$, we have
\begin{enumerate}
\item $W_p(\mu,\nu)\ge \omega(\mu,\nu)^{1+\frac 1p}.$
\item A subset $\mathcal M\subset\mathcal P(\mathbb R^d)$ is compact in $W_p$ if and only if
$\mathcal M$ is weakly compact, and
\begin{align*}
\lim_{N\to\infty}\sup_{\mu\in\mathcal M}\int_{\{x:|x-x_0|>N\}}|x-x_0|^p\mu(dx)=0,
\end{align*}
for some $x_0\in\mathbb R^d$.
\end{enumerate}
\end{prop}

Let $C_{b}(\mathbb R^+\times\mathbb R^d\times\mathcal P(\mathbb R^d))$ denote the space of continuous and bounded functions on $\mathbb R^+\times\mathbb R^d\times\mathcal P(\mathbb R^d)$,
$\mathcal B(\mathbb R^+\times\mathbb R^d\times\mathcal P(\mathbb R^d))$ denote the Borel $\sigma$-algebra on $\mathbb R^+\times\mathbb R^d\times\mathcal P(\mathbb R^d)$.

\begin{de}[Lions derivative]
A function $f:\mathcal P_2(\mathbb R^d)\to\mathbb R$ is called differentiable at $\mu\in\mathcal P_2(\mathbb R^d)$, denoted by $\partial_{\mu} f$,
if there exists a random variable $X\in L^2(\Omega)$ with $\mu=\mathcal L_X$
such that $F(X):=f(\mathcal L_X)$ and $F$ is Fr\'echet differentiable at $X$.
$f$ is called differentiable on $\mathcal P_2(\mathbb R^d)$ if $f$ is differentiable at any $\mu\in\mathcal P_2(\mathbb R^d)$.
\end{de}
\begin{de}
The space $C^{(1,1)}(\mathcal P_2(\mathbb R^d))$ contains the functions $f:\mathcal P_2(\mathbb R^d)\to\mathbb R$
satisfying the following conditions:
(i) $f$ is differentiable and its derivative $\partial_{\mu} f(\mu)(y)$ has a jointly continuous version in $(\mu,y)$, still denoted $\partial_{\mu} f(\mu)(y)$;
(ii) $\partial_{\mu} f(\mu)(\cdot)$ is continuously differentiable for any $\mu$,
and its derivative $\partial_y \partial_{\mu} f(\mu)(y)$ is jointly continuous at any $(\mu,y)$.
\end{de}

\section{Periodic Markov process}
In this section, we consider the existence, convergence and continuous dependence on parameters of periodic Markov process
with state space $\mathbb R^d\times\mathcal P(\mathbb R^d)$.
Before giving the main results in this section,
we first define periodic Markov process with state space $\mathbb R^d\times\mathcal P(\mathbb R^d)$
and the corresponding transition probability function.
We then obtain the equivalent conditions for the existence of the periodic Markov process.
Futhermore, we illustrate the convergence and continuous dependence on parameters of periodic Markov process.
So we divide this section into three parts:
Existence, Convergence and Continuous dependence.

\subsection{Existence}
In this subsection, we investigate the equivalent conditions for the existence of periodic Markov processes
defined on probability space $(\Omega,\mathcal F,P)$ with state space $\mathbb R^d\times\mathcal P(\mathbb R^d)$.
Before giving the main results in this subsection,
we first define the so-called stochastic process and Markov process with state space $\mathbb R^d\times\mathcal P(\mathbb R^d)$
and its transition probability.
\begin{de}
A family $\{M_t:t\in [t_0,R]\}$ of $\mathbb R^d\times\mathcal P(\mathbb R^d)$-valued random variables
is called a stochastic process with state space $\mathbb R^d\times\mathcal P(\mathbb R^d)$.
\end{de}
\begin{de}
A stochastic process $\{M_t:t\in [t_0,R]\}$ defined on probability space $(\Omega,\mathcal F,\mathbb P)$
with index $[t_0,R]\subset [0,\infty)$ and state space $\mathbb R^d\times\mathcal P(\mathbb R^d)$
is called a Markov process with state space $\mathbb R^d\times\mathcal P(\mathbb R^d)$
if it satisfies the following Markov property:
$$\mathbb P(M_t\in A|\mathcal F([s,u]))=\mathbb P(M_t\in A|M_{u})~\quad a.s.$$
Here $\mathcal F([s,u]):=\sigma(M_t:t\in [s,u])$ denotes the smallest $\sigma$-algebra such that $M_t$ is measurable for any $t\in [s,u]$.
\end{de}
\begin{de}
A function $P(s,(x,\mu),t,A)$ is called a transition probability function if the followings hold:
\begin{itemize}
\item[(i)] $P(s,(x,\mu),t,\cdot)$ is a probability on $\mathcal B(\mathbb R^d\times\mathcal P(\mathbb R^d))$ for any $s\le t,(x,\mu)\in\mathbb R^d\times\mathcal P(\mathbb R^d)$;

\item[(ii)] $P(s,\cdot,t,A)$ is $\mathcal B(\mathbb R^d\times\mathcal P(\mathbb R^d))$-measurable for any $s\le t,A\in\mathcal B(\mathbb R^d\times\mathcal P(\mathbb R^d))$;

\item[(iii)] For any $t_0\le s\le u\le t\le R,(x,\mu)\in\mathbb R^d\times\mathcal P(\mathbb R^d)$ and $A\in\mathcal B(\mathbb R^d\times\mathcal P(\mathbb R^d))$, we have the Chapman-Kolmogorov equation
    $$P(s,(x,\mu),t,A)=\int_{\mathbb R^d\times\mathcal P(\mathbb R^d)}P(s,(x,\mu),u,(dy,d\nu))\cdot P(u,(y,\nu),t,A);$$

\item[(iv)] For any $s\in [t_0,R]$ and $A\in\mathcal B(\mathbb R^d\times\mathcal P(\mathbb R^d))$, we have
            $$P(s,(x,\mu),s,A)=1_{A}(x,\mu)=\left\{
                                             \begin{aligned}
                                              & 1, && {(x,\mu)\in A},\\
                                              & 0, && {\rm else}.
                                             \end{aligned}
                                             \right.$$
\end{itemize}
\end{de}
\begin{de}
$P(s,(x,\mu),t,A)$ is called a transition probability function of the Markov process $M_t$ with state space $\mathbb R^d\times\mathcal P(\mathbb R^d)$ if
$$P(s,M_s,t,A)=\mathbb P(M_t\in A|M_s)~~a.s.$$
\end{de}
Similar with the general Markov process with state space $\mathbb R^d$,
all finite-dimensional distributions of the process can be obtained from the transition probability and its initial distribution.
The proof is similar to the classical case, so we omit.
\begin{prop}\label{prop1}
If $M_t$ is a Markov process with state space $\mathbb R^d\times\mathcal P(\mathbb R^d)$, transition probability $P(s,(x,\mu),t,A)$ and initial distribution $\mu$,
we have
\begin{align*}
\mathbb P(M_{t_1}\in A_1,...,M_{t_n}\in A_n)=&\int_{\mathbb R^d}\int_{A_1}\cdots\int_{A_{n-1}}P(t_{n-1},(x_{n-1},\mu_{n-1}),t_n,A_n)\\
                                              &\quad\cdots P(t_0,(x_0,\mu_0),t,(dx_1,d\mu_1))\mu(dx_0,d\mu_0)
\end{align*}
for any $t_0\le t_1<...<t_n\le T, A_1,...,A_n\in\mathcal B(\mathbb R^d\times\mathcal P(\mathbb R^d))$.
Particularly,
$$\mathbb P(M_t\in A)=\int_{\mathbb R^d}P(t_0,(x_0,\mu_0),t,A)\mu(dx_0,d\mu_0).$$
\end{prop}
Conversely, for given transition probability and initial distribution, there exists
a Markov process with state space $\mathbb R^d\times\mathcal P(\mathbb R^d)$ which admits the given transition probability and initial distribution;
for details, see \cite[Theotem 1.5]{RY} for instance.
\begin{prop}
Given a transition probability function $P$ and any probability $\nu$ on $\mathbb R^d\times\mathcal P(\mathbb R^d)$,
there exists a Markov process with state space $\mathbb R^d\times\mathcal P(\mathbb R^d)$, transition probability function $P$
and initial probability $\nu$.
\end{prop}

We then utilize the transition probability function $P$ to construct a family of operators $\{P_t\}_{t\ge t_0}$ as follows:
\begin{equation}\label{DP}
P_tf(s,x,\mu)=\int P(s,(x,\mu),s+t,(dy,d\nu))f(s+t,y,\nu)
\end{equation}
for any $f\in \mathcal B_b(\mathbb R^+\times\mathbb R^d\times\mathcal P(\mathbb R^d))$,
which is indeed a semigroup by Proposition \ref{prop1}, i.e. $P_{t+s}=P_t\circ P_s$.
Indeed, we have
\begin{align*}
P_{t+s}f(t_0,x,\mu)
=&
\int P(t_0,(x,\mu),t_0+t+s,(dy,d\nu))f(t_0+t+s,y,\nu)\\
=&
\int \int P(t_0,(x,\mu),t_0+t,(dz,dm))P(t_0+t,(z,m),t_0+t+s,(dy,d\nu))\\
&\quad\cdot f(t_0+t+s,y,\nu)\\
=&
\int P(t_0,(x,\mu),t_0+t,(dz,dm))P_sf(t_0+t,z,m)\\
=&
P_t\circ P_sf(t_0,x,\mu).
\end{align*}
$P(s,(x,\mu),t,A)$ is called a Feller transition probability function
if $P_tf\in C_b(\mathbb R^+\times\mathbb R^d\times\mathcal P(\mathbb R^d))$ for any $f\in C_b(\mathbb R^+\times\mathbb R^d\times\mathcal P(\mathbb R^d))$.

Here we explain why we need to introduce Markov process with state space $\mathbb R^d\times\mathcal P(\mathbb R^d)$ once more.
Note that if we only consider space $\mathbb R^d$, the Fokker-Planck equation for MVSDE is nonlinear,
and the operator $P_tf(s,x):=\int P(s,x,s+t,dy)f(s+t,y)$ is not a semigroup,
therefore the Krylov-Bogolioubov Theorem is invalid.
However if consider space $\mathbb R^d\times\mathcal P(\mathbb R^d)$, $P_t$ defined in \eqref{DP} is a semigroup and linear,
therefore the Krylov-Bogolioubov Theorem is valid.

We next give the definitions of periodic transition probability function and periodic Markov process.
\begin{de}
The transition probability function $P$ is said to be $T$-periodic
if it satisfies $P(s+T,(x,\mu),t+T,A)=P(s,(x,\mu),t,A)$ for any $s,t,(x,\mu),A$.
\end{de}
\begin{de}[Periodic Markov process]
The Markov process $\{M_t\}$ with $t\in[t_0,R]$ and state space $\mathbb R^d\times\mathcal P(\mathbb R^d)$
is called periodic Markov process with period $T$
if its transition probability function is $T$-periodic
and its initial distribution $P_0(t,A)=P(M_t\in A)$ satisfies
$$P_0(t,A)
=
\int P_0(t,(dx,d\mu))P(t,(x,\mu),t+T,A)=P_0(t+T,A)$$
for any $A\in\mathcal B(\mathbb R^d\times\mathcal P(\mathbb R^d))$.
Specially, if $P_0$ is independent of $t$,
then $\{M_t\}$ is called a stationary Markov process.
\end{de}

We then give one of the main results in this subsection, which describes the equivalent condition for the existence of periodic Markov process with state space $\mathbb R^d\times\mathcal P(\mathbb R^d)$.
\begin{thm}\label{main1}
There exists a $T$-periodic Markov process with state space $\mathbb R^d\times\mathcal P(\mathbb R^d)$ and a given $T$-periodic Feller transition probability function $P(s,(x,\mu),t,A)$
if and only if
\begin{equation}\label{Eq3.1}
\lim_{R\to\infty}\liminf_{n\to\infty}\frac 1n\sum_{k=1}^n P(s_0,(x_0,\mu_0),s_0+kT,U_R^c)=0
\end{equation}
for some $s_0\ge 0,(x_0,\mu_0)\in\mathbb R^d\times\mathcal P(\mathbb R^d)$, where $U_R^c:=\{(x,\mu):|x|\vee|\mu|_2> R\},|\mu|_2^2:=\int |x|^2\mu(dx)$.
\end{thm}
\begin{proof}
Necessity.
Suppose that equality \eqref{Eq3.1} does not hold,
that is to say that
there exists a function $g$ with $g(s,(x,\mu),T)>0$ for any $s,(x,\mu)$ such that
$$\lim_{R\to\infty}\liminf_{n\to\infty}\frac 1n\sum_{k=1}^n P(s,(x,\mu),s+kT,U_R^c)=g(s,(x,\mu),T)>0.$$
Assume that the distribution of the $T$-periodic Markov process is $P_0$,
which is $T$-periodic, i.e. $P_0(t+T,\cdot)=P_0(t,\cdot)$.
Note that we have
\begin{align*}
&\frac 1n\sum_{k=1}^n \int P(s,(x,\mu),s+kT,U_R^c)P_0(s,(dx,d\mu))\\
=&
\frac 1n\sum_{k=1}^n P_0(s+kT,U_R^c)\\
=&
P_0(s,U_R^c),
\end{align*}
and
\begin{align*}
&\lim_{R\to\infty}\liminf_{n\to\infty}\frac 1n\sum_{k=1}^n \int P(s,(x,\mu),s+kT,U_R^c)P_0(s,(dx,d\mu))\\
\ge &
\int g(s,(x,\mu),T)P_0(s,(dx,d\mu))
>0,
\end{align*}
where the first inequality holds by Fatou's lemma.
Combining the above two formulas, we obtain
$$0
=
\lim_{R\to\infty}P_0(s,U_R^c)
\ge
\int g(s,(x,\mu),T)P_0(s,(dx,d\mu)
>0,$$
which is a contradiction.
Therefore, we get
$$\lim_{R\to\infty}\liminf_{n\to\infty}\frac 1n\sum_{k=1}^n P(s_0,(x_0,\mu_0),s_0+kT,U_R^c)=0$$
for some $s_0,(x_0,\mu_0)$.

Sufficiency.
By limit \eqref{Eq3.1}, we know that for some $s_0,(x_0,\mu_0)$, there exists a subsequence $\{n_k\}\subset \mathbb N$ with $n_k\to\infty$ as $k\to\infty$
such that
$$\frac 1{n_k}\sum_{\ell=1}^{n_k} P(s_0,(x_0,\mu_0),s_0+\ell T,U_R^c)\to 0$$
uniformly in $k$ as $R\to\infty$.
Define $P_{n_k}(A):=\frac 1{n_k}\sum_{\ell=1}^{n_k} P(s_0,(x_0,\mu_0),s_0+\ell T,A)$.
Then $\{P_{n_k}\}$ is tight,
that is to say that by Prohorov theorem there exists a subsequence, still denoted $\{P_{n_k}\}$,
such that $P_{n_k}$ weakly converges to some probability distribution $P_0$.
We claim that $P_0$ is a $T$-periodic probability.
Indeed, we have
\begin{align*}
&\int P_0(s_0,(dx,d\mu))\int P(s_0,(x,\mu),s_0+T,(dy,d\nu))f(s_0+T,y,\nu)\\
=&
\lim_{k\to\infty}\frac{1}{n_k}\sum_{\ell=1}^{n_k}\int P(s_0,(x_0,\mu_0),s_0+\ell T,(dx,d\mu))\\
&\quad\cdot\int P(s_0+\ell T,(x,\mu),s_0+(\ell+1) T,(dy,d\nu))f(s_0+T,y,\nu)\\
=&
\lim_{k\to\infty}\frac{1}{n_k}\sum_{\ell=1}^{n_k}\int P(s_0,(x_0,\mu_0),s_0+(\ell+1) T,(dy,d\nu))f(s_0+T,y,\nu)\\
=&
\lim_{k\to\infty}\frac{1}{n_k}\sum_{\ell=1}^{n_k}\int P(s_0,(x_0,\mu_0),s_0+\ell T,(dy,d\nu))f(s_0+\ell T,y,\nu)\\
&+\lim_{k\to\infty}\frac{1}{n_k}\int P(s_0,(x_0,\mu_0),s_0+(n_k+1)T,(dy,d\nu))f(s_0+(n_k+1)T,y,\nu)\\
&-\lim_{k\to\infty}\frac{1}{n_k}\int P(s_0,(x_0,\mu_0),s_0+ T,(dy,d\nu))f(s_0+ T,y,\nu)\\
=&
\int P_0(s_0,(dy,d\nu))f(s_0,y,\nu)
\end{align*}
for any $T$-periodic function $f\in C_{b}(\mathbb R^+\times\mathbb R^d\times\mathcal P(\mathbb R^d))$,
where the first equality holds by the Feller property of transition probability function $P$,
the second equality holds by Chapman-Kolmogorov equation,
and the last equality holds by the tight of $\{P_{n_k}\}$ as well as the boundedness of $f$.
Therefore, we have $P_0(s,\cdot)=P_0(s+T,\cdot)$ for any $s$.
\end{proof}

We next give a equivalent condition of that of Theorem \ref{main1}, which is a more easily tested condition.
\begin{thm}\label{main2}
The limit \eqref{Eq3.1} in Theorem \ref{main1} can be replaced by the more easily tested function
\begin{equation}\label{Eq3.2}
\lim_{R\to\infty}\liminf_{N\to\infty}\frac 1N\int_0^N P(s,(x,\mu),s+t,U_R^c)dt=0
\end{equation}
for some $s,(x,\mu)$, provided that the transition probability function satisfies the following assumption:
\begin{equation}\label{Eq3.3}
\alpha(R)=\sup_{(x,\mu)\in U_{\beta(R)},0\le s,t\le T}P(s,(x,\mu),s+t,U_R^c)\to 0 ~\quad as ~R\to\infty
\end{equation}
for some function $\beta(R)$ satisfying $\lim_{R\to\infty}\beta(R)=\infty.$
\end{thm}

\begin{proof}
We first prove that the limits \eqref{Eq3.2} and \eqref{Eq3.3} imply \eqref{Eq3.1}.
For any $u\in ((k-1)T,kT)$,
we have
\begin{align*}
&P(s,(x,\mu),s+kT,U_R^c)\\
={}&
\bigg(\int_{U^c_{\beta(R)}}+\int_{U_{\beta(R)}}\bigg)P(s,(x,\mu),s+u,(dy,d\nu))P(s+u,(y,\nu),s+kT,U_R^c)\\
\le{} &
P(s,(x,\mu),s+u,U^c_{\beta(R)})
+
\sup_{(y,\nu)\in U_{\beta(R)},u\in ((k-1)T,kT)}P(s+u,(y,\nu),s+kT,U_R^c).
\end{align*}
Integrating with respect to $u$ from $(k-1)T$ to $kT$ on the both sides in the above inequality, we obtain
$$P(s,(x,\mu),s+kT,U_R^c)
\le
\frac 1T\int_{(k-1)T}^{kT}P(s,(x,\mu),s+u,U^c_{\beta(R)})du
+
\alpha(R).$$
Thus, we get
\begin{align*}
\frac 1n\sum_{k=1}^nP(s,(x,\mu),s+kT,U_R^c)
\le &
\frac 1n\sum_{k=1}^n\frac 1T\int_{(k-1)T}^{kT}P(s,(x,\mu),s+u,U^c_{\beta(R)})du
+
\alpha(R)\\
=&
\frac 1{nT}\int_{0}^{nT}P(s,(x,\mu),s+u,U^c_{\beta(R)})du
+
\alpha(R).
\end{align*}
Taking limit on the both sides in the above inequality,
by limits \eqref{Eq3.1} and \eqref{Eq3.2} we obtain
\begin{align*}
&\lim_{R\to\infty}\liminf_{n\to\infty}\frac 1n\sum_{k=1}^nP(s,(x,\mu),s+kT,U_R^c)\\
\le &
\lim_{R\to\infty}\liminf_{n\to\infty}\frac 1{nT}\int_{0}^{nT}P(s,(x,\mu),s+u,U^c_{\beta(R)})du
+
\lim_{R\to\infty}\alpha(R)\\
= &
0.
\end{align*}
Therefore, we get the desired result, i.e. limit \eqref{Eq3.1} holds.

We then prove that limits \eqref{Eq3.1} and \eqref{Eq3.3} imply \eqref{Eq3.2}.
By Champan-Kolmogorov equation on $\mathbb R^d\times\mathcal P(\mathbb R^d)$, we have
\begin{align*}
&P(s,(x,\mu),s+u,U^c_{R})\\
=&
\bigg(\int_{U^c_{\beta(R)}}+\int_{U_{\beta(R)}}\bigg)P(s,(x,\mu),s+(k-1)T,(dy,d\nu))P(s+(k-1)T,(y,\nu),s+u,U_R^c)\\
\le &
P(s,(x,\mu),s+(k-1)T,U^c_{\beta(R)})
+
\sup_{(y,\nu)\in U_{\beta(R)},u\in ((k-1)T,kT)}P(s+(k-1)T,(y,\nu),s+u,U_R^c)\\
\le &
P(s,(x,\mu),s+(k-1)T,U^c_{\beta(R)})+ \alpha(R)
\end{align*}
for any $u\in ((k-1)T,kT)$.
Integrating with respect to $u$ on $[0,N]$ and summing to $k$ from $1$ to $n$ on the both sides in the above inequality, we obtain
$$\frac 1N\int_0^NP(s,(x,\mu),s+u,U^c_{R})du
\le
\frac 1n\sum_{k=1}^nP(s,(x,\mu),s+(k-1)T,U^c_{\beta(R)})+ \alpha(R).$$
Therefore, by limit \eqref{Eq3.1} and \eqref{Eq3.3} we get
\begin{align*}
&\lim_{R\to\infty}\liminf_{N\to\infty}\frac 1N\int_0^NP(s,(x,\mu),s+u,U^c_{R})du\\
\le &
\lim_{R\to\infty}\liminf_{n\to\infty}\frac 1n\sum_{k=1}^nP(s,(x,\mu),s+(k-1)T,U^c_{\beta(R)})+ \lim_{R\to\infty}\alpha(R)\\
= &
0.
\end{align*}
So limit \eqref{Eq3.2} holds.
The proof is complete.
\end{proof}

As usual, if the transition probability function is time-homogeneous, denoted by $P((x,\mu),t,A)$,
we have the equivalent conditions
for the existence of stationary Markov process with state space $\mathbb R^d\times\mathcal P(\mathbb R^d)$,
which is stated as the following corollary.
Before giving the corollary, we introduce a definition as follows:
a transition probability function $P$ is called stochastically continuous if $P(t_0,(x,\mu),t,U_{\epsilon}(x,\mu))\to 1$ as $t\to t_0$ for any $\epsilon>0$, where $U_{\epsilon}(x,\mu):=\{(y,\nu):|y-x|\vee W_2(\mu,\nu)<\epsilon\}$.
\begin{cor}\label{cor1}
There exists a stationary Markov process with state space $\mathbb R^d\times\mathcal P(\mathbb R^d)$
and with a given time-homogeneous stochastically continuous Feller transition probability function $P((x,\mu),t,A)$
if and only if
$$\lim_{R\to\infty}\liminf_{T\to\infty}\frac 1T\int_{0}^T P((x,\mu),t,U_R^c)dt=0$$
for some $(x,\mu)$.
\end{cor}
\begin{proof}
Necessity.
Suppose that the limit does not hold, i.e.
there exists a function $g$ with $g(x,\mu)>0$ for any $(x,\mu)\in\mathbb R^d\times\mathcal P(\mathbb R^d)$ such that
$$\lim_{R\to\infty}\liminf_{T\to\infty}\frac 1T\int_{0}^T P((x,\mu),t,U_R^c)dt=g(x,\mu)>0.$$
Let $P_0$ denote the distribution of the stationary Markov process,
which is indeed a stationary probability measure.
Then we have
$$P_0(U_R^c)=\int P_0(dx,d\mu)\frac 1T\int_0^T P((x,\mu),t,U_R^c)dt.$$
Thus by Fatou's lemma we obtain
\begin{align*}
0
=&
\lim_{R\to\infty}P_0(U_R^c)\\
=&
\lim_{R\to\infty}\liminf_{T\to\infty}\int P_0(dx,d\mu)\frac 1T\int_0^T P((x,\mu),t,U_R^c)\\
\ge&
\int P_0(dx,d\mu)\lim_{R\to\infty}\liminf_{T\to\infty}\frac 1T\int_0^T P((x,\mu),t,U_R^c)\\
=&
\int P_0(dx,d\mu)g(x,\mu)
>
0,
\end{align*}
which is a contradiction.
Therefore, we obtain that the following equality
$$\lim_{R\to\infty}\liminf_{T\to\infty}\frac 1T\int_{0}^T P((x,\mu),t,U_R^c)dt=0$$
holds for some $(x,\mu)\in\mathbb R^d\times\mathcal P(\mathbb R^d)$.

Sufficiency.
By the limit,
there exists $(x_0,\mu_0),T_n\to\infty$ such that
$$\frac 1{T_n}\int_{0}^{T_n} P((x_0,\mu_0),t,U_R^c)dt\to 0$$
uniformly in $n$ as $R\to\infty$.
Define $P_n(\cdot):=\frac 1{T_n}\int_{0}^{T_n} P((x_0,\mu_0),t,\cdot)dt$.
Thus $\{P_n\}$ is tight.
By Prohorov theorem, there exists a subsequence, still denoted $\{P_n\}$, such that $P_n$ weakly converges to some probability measure $P_0$.
Therefore, $P_0$ is a stationary probability measure.
Indeed, we have
\begin{align*}
&\int P_0(dx,d\mu)\int P((x,\mu),t,(dy,d\nu))f(y,\nu)\\
=&
\lim_{n\to\infty}\frac{1}{T_n}\int_0^{T_n} ds\int P((x_0,\mu_0),s,(dx,d\mu))\int P((x,\mu),t,(dy,d\nu))f(y,\nu)\\
=&
\lim_{n\to\infty}\frac{1}{T_n}\int_0^{T_n} ds \int P((x_0,\mu_0),s+t,(dy,d\nu))f(y,\nu)\\
=&
\lim_{n\to\infty}\frac{1}{T_n}\int_0^{T_n} du \int P((x_0,\mu_0),u,(dy,d\nu))f(y,\nu)\\
&+\lim_{n\to\infty}\frac{1}{T_n}\int_{T_n}^{T_n +t} du \int P((x_0,\mu_0),u,(dy,d\nu))f(y,\nu)\\
&-\lim_{n\to\infty}\frac{1}{T_n}\int_0^{t} du \int P((x_0,\mu_0),u,(dy,d\nu))f(y,\nu)\\
=&
\lim_{n\to\infty}\frac{1}{T_n}\int_0^{T_n} du \int P((x_0,\mu_0),u,(dy,d\nu))f(y,\nu)\\
=&
\int P_0(dy,d\nu)f(y,\nu)
\end{align*}
for any $f\in C_{b}(\mathbb R^d\times\mathcal P(\mathbb R^d))$.
\end{proof}
\begin{rem}
\begin{enumerate}
\item Although the time-independent probability is a stationary probability,
Corollary \ref{cor1} could not be directly obtained by Theorem \ref{main1}
since the stochastically continuous of transition probability $P$ could not satisfy limit \eqref{Eq3.3}.
However, we note that the proof of Corollary \ref{cor1} is similar with that of Theorem \ref{main1}.

\item It is worth noting that the proof of existence of stationary probability is indeed the Krylov-Bogolioubov theorem on space $\mathbb R^d\times \mathcal P(\mathbb R^d)$.
\end{enumerate}
\end{rem}

\subsection{Convergence}
In this subsection, we consider the convergence of periodic Markov process,
which is exactly said to be the convergence of a subsequence to the periodic Markov process.

\begin{thm}\label{main4}
Assume that the $T$-periodic Feller transition probability function $P$ satisfies
$$\lim_{R\to\infty}\liminf_{n\to\infty}\frac 1n\sum_{k=1}^n P(s,(x,\mu),s+kT,U_R^c)=0$$
for any $s,(x,\mu)$.
Then there exist $T$-periodic probability $\mu_s$ and a subsequence $\{n_j\}\subset \mathbb N$ such that
\begin{equation}\label{Eq4.1}
\lim_{n_j\to\infty}\frac 1{n_jT}\int_t^{t+n_jT}\int_{\mathbb R^d\times \mathcal P(\mathbb R^d)}f(s,x,\mu)\nu_s(dx,d\mu)ds
=
\frac 1T\int_0^T\int_{\mathbb R^d\times \mathcal P(\mathbb R^d)} f(s,x,\mu)\mu_s(dx,d\mu)ds
\end{equation}
for any probability measure $\{\nu_s\}$, $T$-periodic $f\in C_b(\mathbb R^+\times \mathbb R^d\times\mathcal P(\mathbb R^d))$.
In particular, if $\mu_s$ is a unique periodic probability,
then the convergence in above holds for the whole sequence $\mathbb N$.
\end{thm}
\begin{proof}
By the condition there exists a subsequence $\{n_j\}\subset \mathbb N$ with $\lim_{j\to\infty}n_j=\infty$ such that
$$\frac 1{n_j}\sum_{k=1}^{n}P(s,(x,\mu),s+kT,U_R^c)\to 0$$
uniformly in $j$ as $R\to\infty$.
Define $\nu_s^n:= \frac 1n\sum_{k=1}^{n}\nu_{s+kT}$.
Then we have
\begin{align*}
\lim_{R\to\infty}\nu_s^{n_j}(U_R)
= &
\lim_{R\to\infty}\frac 1{n_j}\sum_{k=1}^{n_j}\nu_{s+kT}(U_R)\\
= &
\lim_{R\to\infty}\frac 1{n_j}\sum_{k=1}^{n_j}\int \nu_{s}(dx,d\mu)P(s,(x,\mu),s+kT,U_R)\\
\ge &
\int \nu_{s}(dx,d\mu) \lim_{R\to\infty}\frac 1{n_j}\sum_{k=1}^{n_j}P(s,(x,\mu),s+kT,U_R)\\
= &
1.
\end{align*}
Thus, $\{\nu_s^{n_j}\}$ is tight uniformly in time,
i.e. there exists a subsequence, still denoted $\{\nu_s^{n_j}\}$,
such that $\nu_s^{n_j}$ weakly converges to some probability measure $\mu_s$ uniformly in time.
And we claim that $\mu_s$ is $T$-periodic.
Indeed, we have
\begin{align*}
\mu_{s+T}(A)
= &
\int P(s,(x,\mu),s+T,A)\mu_s(dx,d\mu)\\
= &
\lim_{j\to\infty}\int P(s,(x,\mu),s+T,A)\nu^{n_j}_s(dx,d\mu)\\
= &
\lim_{j\to\infty}\frac 1{n_j}\sum_{k=1}^{n_j} \int P(s,(x,\mu),s+T,A)\nu_{s+k T}(dx,d\mu)\\
= &
\lim_{j\to\infty}\frac 1{n_j}\sum_{k=1}^{n_j} \int P(s+kT,(x,\mu),s+(k+1)T,A) \int P(s,(y,\nu),s+kT,(dx,d\mu))\nu_s(dy,d\nu)\\
= &
\lim_{j\to\infty}\frac 1{n_j}\sum_{k=1}^{n_j} \int P(s,(y,\nu),s+(k+1)T,A)\nu_s(dy,d\nu)\\
= &
\lim_{j\to\infty}\frac 1{n_j}\sum_{k=1}^{n_j} \int P(s,(y,\nu),s+kT,A)\nu_s(dy,d\nu)\\
&-
\lim_{j\to\infty}\frac 1{n_j}\int P(s,(y,\nu),s+T,A)\nu_s(dy,d\nu)\\
&+
\lim_{j\to\infty}\frac 1{n_j}\int P(s,(y,\nu),s+(n_j+1) T,A)\nu_s(dy,d\nu)\\
= &
\lim_{j\to\infty}\frac 1{n_j}\sum_{k=1}^{n_j} \int P(s,(y,\nu),s+kT,A)\nu_s(dy,d\nu)\\
= &
\lim_{j\to\infty}\nu_s^{n_j}(A)
=
\mu_s(A)
\end{align*}
for any $A\in\mathcal B(\mathbb R^d)$.

We next prove equality \eqref{Eq4.1}. We have
\begin{align*}
&\lim_{j\to\infty}\frac 1{n_jT}\int_t^{t+n_jT}\int_{\mathbb R^d\times \mathcal P(\mathbb R^d)}f(s,x,\mu)\nu_s(dx,d\mu)ds\\
=&
\lim_{j\to\infty}\frac 1{n_jT}(\int_t^{t+\epsilon}+\int_{t+\epsilon}^{t+\epsilon+n_jT}+\int_{t+\epsilon+n_jT}^{t+n_jT})\int_{\mathbb R^d\times \mathcal P(\mathbb R^d)}f(s,x,\mu)\nu_s(dx,d\mu)ds\\
=&
\lim_{j\to\infty}\frac 1{n_jT}\int_{t+\epsilon}^{t+\epsilon+n_jT}\int_{\mathbb R^d\times \mathcal P(\mathbb R^d)}f(s,x,\mu)\nu_s(dx,d\mu)ds\\
=&
\lim_{j\to\infty}\frac 1T\int_{t+\epsilon}^{t+\epsilon+T}\int_{\mathbb R^d\times \mathcal P(\mathbb R^d)}f(s,x,\mu)\nu_s^{n_j}(dx,d\mu)ds\\
=&
\frac 1T\int_{t+\epsilon}^{t+\epsilon+T}\int f(s,x,\mu)\mu_s(dx,d\mu)ds\\
=&
\frac 1T\int_{t}^{t+T}\int f(s,x,\mu)\mu_s(dx,d\mu)ds,
\end{align*}
for any $T$-periodic function $f\in C_b(\mathbb R^+\times \mathbb R^d\times \mathcal P(\mathbb R^d))$,
where the second equality holds by the boundedness of $f$,
the third equality holds due to $T$-period of function $f$,
the fourth equality holds by the definition of $\mu^n$,
and the last equality holds by the $T$-period of function $f$ as well as $\mu_s$.

If $T$-periodic probability $\mu_s$ is unique, it is obvious that equality \eqref{Eq4.1} holds for the whole sequence $\mathbb N$.
The proof is complete.
\end{proof}

\subsection{Continuous dependence}
In this subsection, we investigate the continuous dependence on parameters for periodic Markov processes with state space $\mathbb R^d\times\mathcal P(\mathbb R^d)$.

\begin{thm}\label{main6}
Assume that the $T$-periodic Feller transition probability function family $\{P^k\}_{k\in\mathbb N}$ satisfy
$$\lim_{R\to\infty}\liminf_{n\to\infty}\frac 1n\sum_{k=1}^n P^k(s_0,(x_0,\mu_0),s_0+kT,U_R^c)=0$$
for some $s_0,(x_0,\mu_0)$,
and $$\lim_{k\to\infty}P^k(s,(x,\mu),t,A)=P(s,(x,\mu),t,A)$$
for any $t\ge s,(s,\mu),A$.
Then the subsequence of $\{P_0^k\}_{k\in\mathbb N}$ converges to $P_0$,
where $P_0^k,P_0$ denote the $T$-periodic probability for the transition function $P^k,P$, respectively.
\end{thm}
\begin{proof}
By Theorem \ref{main1}, we get a sequence of $T$-periodic probability measures $\{P_0^k\}_{k\in\mathbb N}$,
and $P_0^k(s,U_R)>1-\epsilon$ for any $k,0<\epsilon<1$ if $R$ is sufficiently large.
That is to say that $\{P_0^k\}_{k\in\mathbb N}$ is tight uniformly in time.
Therefore, there exists a subsequence, still denoted $\{P_0^k\}_{k}$,
such that $P_0^k$ weakly converges to some probability measure $P_0$ uniformly in time.
We claim that $P_0$ is the $T$-periodic probability for the transition probability function $P$.
Indeed, we have
\begin{align*}
&\int P_0(s,(dx,d\mu))\int P(s,(x,\mu),s+T,(dy,d\nu))f(s+T,y,\nu)\\
=&
\int f(s+T,y,\nu)\int P_0(s,(dx,d\mu))P(s,(x,\mu),s+T,(dy,d\nu))\\
=&
\lim_{k\to\infty}\int f(s+T,y,\nu)\int P^k_0(s,(dx,d\mu))P^k(s,(x,\mu),s+T,(dy,d\nu))\\
=&
\lim_{k\to\infty}\int f(s+T,y,\nu)P^k_0(s+T,(dy,d\nu))\\
=&
\int f(s+T,y,\nu)P_0(s+T,(dy,d\nu))
\end{align*}
for any $T$-periodic function $f\in C_b(\mathbb R^+\times\mathbb R^d\times\mathcal P(\mathbb R^d))$.
\end{proof}

\section{Periodic solutions for MVSDEs}
In this section, we investigate the existence, convergence and continuous dependence on parameters of periodic solutions
for MVSDEs under periodic distribution-dependent Lyapunov conditions,
which is obtained by the periodic Markov processes with state space $\mathbb R^d\times\mathcal P(\mathbb R^d)$.
To obtain the main results in this section, we first simultaneously truncate both space and distribution variables
of the coupled MVSDEs' coefficients,
and utilize the periodic Lyapunov conditions to get the existence of periodic probability
for transition probability function generated by the coupled MVSDEs.
We then utilize the results about coupled MVSDEs to obtain the existence of periodic solutions for MVSDEs.
Additionally, we study the convergence and continuous dependence on parameters of periodic solutions for MVSDEs under periodic Lyapunov conditions,
which is also obtained in the sense of the convergence and continuous dependence of periodic Markov processes.
So, we devide this section into three subsection:
Existence, Convergence and Continuous dependence.

\subsection{Existence}
In this subsection, we apply Theorem \ref{main2} for McKean-Vlasov SDEs to obtain the existence of periodic solutions
under periodic distribution-dependent Lyapunov conditions.
However, the initial value $(x,\mu)$ for the transition probability
of the Markov process with state space $\mathbb R^d\times\mathcal P(\mathbb R^d)$
is not a `pair', that is to say that $\mu$ is arbitrarily chosen from $\mathcal P(\mathbb R^d)$,
and may not equal to $\delta_x$ which is a delta measure,
so here we first need to deal with coupled McKean-Vlaosv SDEs
\begin{align}\label{CMVSDE}
\left\{
\begin{aligned}
&dX_t= b(t,X_t,\mathcal{L}_{X_t})dt+ \sigma(t,X_t,\mathcal{L}_{X_t})dW_t,~X_s=\xi,\\
&d\bar X_t= \bar b(t,\bar X_t,\mathcal{L}_{X_t})dt+ \bar \sigma(t,\bar X_t,\mathcal{L}_{X_t})dW_t,~\bar X_s=x,
\end{aligned}
\right.
\end{align}
where $\mathcal L_{\xi}=\mu$, then study the existence of periodic solutions for MVSDEs.

If the coefficients of the coupled MVSDEs are locally Lipschitz and satisfy locally linear growth conditions,
then we generally utilize the truncation method of coefficients to obtain the existence of solutions as follows:
for any $n\geq 1, t\in [0,T],x\in C[0,T]$ and $\mu \in \mathcal P(C[0,T])$, define
\begin{align*}
b^n(t, x_t, \mu_t):= b(t, \phi_n(x_t), \mu_t \circ \phi_n),
\sigma^n(t, x_t, \mu_t):= \sigma(t, \phi_n(x_t), \mu_t \circ \phi_n),
\end{align*}
where $\phi_n(x_t):=x_{t\wedge \tau_x^n},\tau_x^n:=\inf\{t\ge s: |x_t|\ge n\}$.
Thus for each $n\geq 1$, $b^n, \sigma^n$ are Lipschitz and satisfy the linear growth condition.
Therefore, by \cite[Theorem 2.1]{Wang_18}, the equation
\begin{equation}\label{TMVSDE}
\left\{
\begin{aligned}
dX^n_t&= b^n(t,X^n_t,\mathcal{L}_{X^n_t})dt+ \sigma^n(t,X^n_t,\mathcal{L}_{X^n_t})dW_t\\
X^n_s &= X_s \\
\end{aligned}
\right.
\end{equation}
has a unique solution $X^n$ with finite second-moment.
Define $\tau^n:= \inf\{ t\geq s: |X^n_t|\geq n \}$.
By the definition of $\phi_n$, we get
\begin{align*}
\phi_n(X ^n_t)= X^n_{t\wedge \tau^n_{X^n}}= X^n_{t\wedge \tau^n},
\end{align*}
and for every measurable set $A \subset \mathbb R^d$, we have
\begin{align*}
&{(\mathcal{L}_{X^n_t}\circ \phi_n^{-1}})(A)= P(X^n_t \in \phi_n^{-1}(A))
={} P(\phi_n(X^n_t)\in A)=\mathcal{L}_{\phi_n (X^n_t)}(A)= \mathcal{L}_{X^n_{t\wedge \tau^n}}(A).
\end{align*}
Hence the equation \eqref{TMVSDE} is equivalent to the following SDE:
\begin{equation}\label{TEMVSDE}
\left\{
\begin{aligned}
dX^n_t&= b(t,X^n_{t\wedge \tau^n},\mathcal{L}_{X^n_{t\wedge \tau^n}})dt+ \sigma(t,X^n_{t\wedge \tau^n},\mathcal{L}_{X^n_{t\wedge \tau^n}})dW_t\\
X^n_s &= X_s. \\
\end{aligned}
\right.
\end{equation}
Therefore, the stopped process $Y^n_t:=X^n_{t\wedge \tau^n}$ satisfies the following equation:
\begin{align}\label{SMVSDE}
\left\{
\begin{aligned}
dY^n_t
={} &
1_{[s,\tau^n]}(t)b(t,Y^n_t,\mathcal{L}_{Y^n_t})dt
 +
1_{[s,\tau^n]}(t)\sigma(t,Y^n_t,\mathcal{L}_{Y^n_t})dW_t\\
Y^n_s = {}& X_s, \\
\end{aligned}
\right.
\end{align}
where $1$ denotes the indicator function.
Similarly, define $\bar X_t^n$ by the following equation:
\begin{equation}\label{TEMVSDE}
\left\{
\begin{aligned}
d\bar X^n_t&= \bar b(t,\bar X^n_{t\wedge \tau^n},\mathcal{L}_{X^n_{t\wedge \tau^n}})dt+ \bar \sigma(t,\bar X^n_{t\wedge \tau^n},\mathcal{L}_{X^n_{t\wedge \tau^n}})dW_t\\
\bar X^n_s &= \bar X_s. \\
\end{aligned}
\right.
\end{equation}
Then the stopped process $\bar Y^n_t:=\bar X^n_{t\wedge \tau^n}$ satisfies the following equation:
\begin{align}\label{SMVSDE}
\left\{
\begin{aligned}
d\bar Y^n_t
={} &
1_{[s,\tau^n]}(t)\bar b(t,\bar Y^n_t,\mathcal{L}_{Y^n_t})dt
 +
1_{[s,\tau^n]}(t)\bar \sigma(t,\bar Y^n_t,\mathcal{L}_{Y^n_t})dW_t\\
\bar Y^n_s = {}& \bar X_s.\\
\end{aligned}
\right.
\end{align}

We note that $\tau^n\to\infty,\mathcal L_{Y^n_t}\to\mathcal L_{X_t},\mathcal L_{\bar Y^n_t}\to\mathcal L_{\bar X_t}$ as $n\to\infty$
if the coefficients satisfy the conditions in \cite{LM2}, where $X_t,\bar X_t$ are the solutions for the coupled MVSDEs.
However, in this paper for simplicity we do not introduce the conditions in \cite{LM2},
and directly define that the solutions of equations \eqref{CMVSDE} is called regular at $(\xi,x)$ if $\tau^n\to\infty,\mathcal L_{Y^n_t}\to\mathcal L_{X_t},\mathcal L_{\bar Y^n_t}\to\mathcal L_{\bar X_t}$ as $n\to\infty$.

We give the following Lyapunov condition,
where the Lyapunov function is periodic in time and simultaneously depends on space and distribution variables,
to obtain the existence of periodic probability for the coupled MVSDEs \eqref{CMVSDE}.
\begin{enumerate}
\item [(H)](Lyapunov condition)
There exists a nonnegative $T$-periodic function $V\in C^{1,2,(1,1)}(\mathbb R^+\times\mathbb R^d \times \mathcal P_2(\mathbb R^d))$ such that
for all $(t,x,\mu)\in [0,\infty)\times\mathbb R^d \times \mathcal P_2(\mathbb R^d)$
\begin{align*}
\sup_{|x|\vee|\mu|_2> R} LV(t,x,\mu)&=-A_R  \to-\infty ~as~R\to\infty, \\
V_R(t):= &\inf_{|x|\vee|\mu|_2> R}V(t,x, \mu)\to \infty ~as~ R\to \infty,
\end{align*}
where \begin{align*}
C^{1,2,(1,1)}(\mathbb R^+\times\mathbb R^d\times \mathcal P_2(\mathbb R^d))
:={}&\{f:\mathbb R^+\times\mathbb R^d\times\mathcal P_2(\mathbb R^d)\to\mathbb R|
f(\cdot,x,\mu)\in C^1(\mathbb R^+) \quad \hbox{for }x,\mu,\\
f(t,\cdot, \mu)\in C^2(\mathbb R^d) \quad \hbox{for }t,\mu,
&\quad f(t,x,\cdot)\in C^{(1,1)}(\mathcal P_2(\mathbb R^d))\quad \hbox{for }t,x\},
\end{align*}
and $LV$ is defined as follows:
\begin{align*}
LV(t,x,\mu)
&:={}
\partial_t V(t,x,\mu)+\bar b(t,x,\mu)\cdot \partial_x V(t,x,\mu)+ \frac 12tr\big((\bar\sigma\bar\sigma^{\top})(t,x,\mu)\cdot\partial_x^2 V(t,x,\mu)\big)\\
+&{}\int \bigg[b(t,y,\mu)\cdot\partial_{\mu} V(t,x,\mu)(y)
 + \frac 12tr\big((\sigma\sigma^{\top})(t,y,\mu)\cdot\partial_y \partial_{\mu} V(t,x,\mu)(y)\big)\bigg]\mu(dy).
\end{align*}

\end{enumerate}

\begin{lem}\label{main3}
Assume that the coupled McKean-Vlasov SDEs \eqref{CMVSDE} have a unique regular solutions for at least one $(\xi,x)$ with finite second-moment,
condition (H) holds and coefficients $b,\bar b,\sigma,\bar\sigma$ are $T$-periodic.
Then there exists a $T$-periodic probability for the transition probability function generated by coupled MVSDEs \eqref{CMVSDE}.
\end{lem}
\begin{proof}
By It\^o's formula, we have
\begin{align}\label{Eq3.4}
EV(t\wedge \tau^n,\bar Y^n_t,\mathcal L_{Y^n_t})-V(s,x,\mu)
=
E\int_s^{\tau_n\wedge t}LV(r, \bar Y^n_r,\mathcal L_{Y^n_r})dr.
\end{align}
On the other hand, by condition (H) the following inequality holds
\begin{align}\label{IEq3.5}
LV(r, \bar Y^n_r,\mathcal L_{Y^n_r})
\le
-1_{\{|\bar {Y}^n_r|\vee|\mathcal L_{Y^n_r}|_2>R\}}A_R
+
\sup_{(y,\nu)}LV(r,y,\nu).
\end{align}
Thus, combining equality \eqref{Eq3.4} and inequality \eqref{IEq3.5}, we obtain that there exist constants $C_1,C_2>0$ such that
\begin{align*}
A_R E\int_s^{\tau_n\wedge t}1_{\{|\bar {Y}^n_r|\vee|\mathcal L_{Y^n_r}|_2>R\}}dr
\le
C_1(t-s)+C_2.
\end{align*}
Note that we have $\tau_n\to\infty,\mathcal L_{Y^n_r}\to \mathcal L_{X_r},\mathcal L_{\bar Y^n_r}\to \mathcal L_{\bar X_r} ~a.s.$ as $n\to\infty$ since there exists a unique regular solution for the coupled MVSDE \eqref{CMVSDE}.
Thus, let $n\to\infty$ and divide $t-s$ on both sides in the above inequality, we get
\begin{align*}
\frac 1{t-s}\int_s^tP(s,(x,\mu),r,U^c_R)dr
\le
\frac{C_3}{A_R}
\to 0 \quad as \quad R\to\infty,
\end{align*}
where $C_3$ is a nonnegative constant.
Let $u=r-s$, we have
\begin{align}\label{Eq3.6}
\frac 1{t-s}\int_0^{t-s}P(s,(x,\mu),s+u,U^c_R)du
\le
\frac{C_3}{A_R}
\to 0 \quad as \quad R\to\infty,
\end{align}
which suggests that limit \eqref{Eq3.2} in Theorem \ref{main2} holds.

We next show that limit \eqref{Eq3.3} in Theorem \ref{main2} holds.
By It\^o's formula, we have
$$EV(t,\bar X_t,\mathcal L_{X_t})-V(s,x,\mu)
=
E\int_s^tLV(r,\bar X_r,\mathcal L_{X_r})dr
\le
\lambda (t-s),$$
where the inequality holds by condition (H) that $LV(t,x,\mu)\le \lambda t$ for a sufficiency large constant $\lambda$.
Therefore, together with Chebyshev's inequality, we get
\begin{align}\label{Eq3.7}
P(s,(x,\mu),t,U_R^c)
\le
\frac{V(s,x,\mu)+\lambda (t-s)}{\inf_{(y,\nu)\in U_R^c}V(r,y,\nu)}.
\end{align}
In summary, combining \eqref{Eq3.6}, \eqref{Eq3.7} and Theorem \ref{main2}, we obtain that there exists a $T$-periodic probability.
\end{proof}
\begin{rem}
Note that the above Theorem could not imply the existence of periodic solution for the coupled MVSDEs
since the corresponding stochastic process with the periodic probability may not have the form of `$\delta_{x}\times\mu$',
which is the initial values' form.
\end{rem}

We then apply Lemma \ref{main3} to obtain the existence of periodic solution for MVSDE,
which is one of the main results in this paper, and described by the following theorem.
\begin{thm}\label{main8}
Assume that the conditions of Lemma \ref{main3} hold if $\bar b,\bar \sigma$ are replaced by $b,\sigma$.
Then there exists a $T$-periodic solution for MVSDE.
\end{thm}
\begin{proof}
By Lemma \ref{main3}, there exists a $T$-periodic probability for the transition probability function $P$ generated by the following SDEs:
\begin{align*}
\left\{
\begin{aligned}
&dX_t= b(t,X_t,\mathcal{L}_{X_t})dt+ \sigma(t,X_t,\mathcal{L}_{X_t})dW_t,~X_s=x,\\
&d\bar X_t= b(t,\bar X_t,\mathcal{L}_{X_t})dt+ \sigma(t,\bar X_t,\mathcal{L}_{X_t})dW_t,~\bar X_s=x,
\end{aligned}
\right.
\end{align*}
Note that we obtain that $X_t=\bar X_t$ for any $t\ge s$,
and the above SDEs' transition probability function $$P(s,(x,\delta_x),t,\cdot)=\mathcal L_{X_t}\times\delta_{\mathcal L_{X_t}}.$$
Let $P_0$ denote the distribution of the $T$-periodic solution,
we have
$$P_0(t+T,A\times\mathcal P(\mathbb R^d))
=
P_0(t,A\times\mathcal P(\mathbb R^d))
$$
for any $A\in\mathcal B(\mathbb R^d)$,
which means that $P_0(\cdot,\cdot\times\mathcal P(\mathbb R^d))$ is a $T$-periodic probability.
We claim that $\pi_t(\cdot):=P_0(t,\cdot\times\mathcal P(\mathbb R^d))$ is a periodic probability for the following McKean-Vlasov SDE:
$$dX_t= b(t,X_t,\mathcal{L}_{X_t})dt+ \sigma(t,X_t,\mathcal{L}_{X_t})dW_t.$$
Indeed,
$\int P(s,(x,\delta_x),t,\cdot)\pi_s(dx)\times\delta_{\pi_s}(d\delta_x)$ is the above MVSDE's distribution at time $t$
with initial distribution $\pi_s\times\delta_{\pi_s}$,
which is obviously a $T$-periodic probability for the above SDEs.
Therefore, $\pi_t$ is a desired $T$-periodic probability for MVSDE,
and there exists a $T$-periodic solution with initial distribution $\pi_s$.
\end{proof}
\begin{rem}
\begin{enumerate}
\item It is worth mentioning that considering distribution-dependent Lyapunov function is natural and reasonable since the coefficients of MVSDEs depend on distribution variable.
The condition (H) is weaker than condition (H3a) or (H3b) in \cite{SW} if the Lyapunov function is independent of distribution.
\end{enumerate}
\end{rem}

As a direct consequence of Theorem \ref{main8}, we have the following corollary,
which gives the existence of stationary solution for time-homogeneous MVSDE.
\begin{cor}
Suppose that the conditions of Theorem \ref{main8} are satisfied if coefficients $b,\sigma$ are independent of $t$.
Then there exists a stationary solution for MVSDE.
\end{cor}

\subsection{Convergence}
In this subsection, we study the convergence of periodic solutions for MVSDEs under periodic Lyapunov conditions,
which is indeed the convergence of a subsequence.
We first apply Theorem \ref{main4} to obtain the convergence of periodic solutions for the coupled MVSDEs,
then show the convergence of periodic solutions for MVSDEs.
\begin{lem}\label{main5}
Assume that the conditions of Theorem \ref{main3} hold,
and there exists a unique regular solutions of the coupled MVSDEs \eqref{CMVSDE} for any $\xi,x$.
Then there exists a subsequence $\{n_j\}\subset \mathbb N$ such that
\begin{equation}\label{Eq4.1}
\lim_{j\to\infty}\frac 1{n_jT}\int_t^{t+n_jT}\int_{\mathbb R^d\times \mathcal P(\mathbb R^d)}f(s,x,\mu)\nu_s(dx,d\mu)ds
=
\frac 1T\int_0^T\int_{\mathbb R^d\times \mathcal P(\mathbb R^d)} f(s,x,\mu)\mu_s(dx,d\mu)ds
\end{equation}
for any probability measure $\nu_s=\nu\times\delta_{\mu}$,
$T$-periodic function $f\in C_b(\mathbb R^+\times \mathbb R^d\times\mathcal P(\mathbb R^d))$,
where $\mu_s$ is a $T$-periodic probability for the equations \eqref{CMVSDE}.
Particularly, if the periodic is unique,
the convergence in above holds for the whole sequence.
\end{lem}
\begin{proof}
It is obviously that the conditions of Theorem \ref{main4} holds by the proof of Lemma \ref{main3}.
Thus by Theorem \ref{main4}, we obtain the desired results.
\end{proof}

As a direct corollary, we have the following result.
\begin{thm}
Assume that the conditions of Theorem \ref{main8} hold,
there exists a unique regular solution of the MVSDE for any $x$.
Then the solutions for MVSDEs weakly converge to its periodic solutions,
which is the weak convergence for some subsequence.
Particularly, if the periodic solution is unique,
the convergence in above holds for the whole sequence.
\end{thm}
\begin{proof}
Let $\nu_s=\nu\times\delta_{\nu}, b=\bar b,\sigma=\bar\sigma$,
by Theorem \ref{main8} and Lemma \ref{main5} we obtain that there exist a $T$-periodic probability $\mu_t$ for MVSDE
and a subsequence $\{n_j\}\subset\mathbb N$ such that
\begin{align*}
&\lim_{j\to\infty}\frac 1{n_jT}\int_t^{t+n_jT}\int_{\mathbb R^d}f(s,x,\nu)\nu(dx)ds\\
=&
\lim_{j\to\infty}\frac 1{n_jT}\int_t^{t+n_jT}\int_{\mathbb R^d\times\mathcal P(\mathbb R^d)}f(s,x,\mu)\nu(dx)\times \delta_{\nu}(d\mu)ds\\
=&
\frac 1T\int_0^T\int_{\mathbb R^d\times\mathcal P(\mathbb R^d)} f(s,x,\mu)\mu_s(dx)\times\delta_{\mu_s}(d\mu)ds\\
=&
\frac 1T\int_0^T\int_{\mathbb R^d} f(s,x,\mu_s)\mu_s(dx)ds
\end{align*}
for any $\nu\in\mathcal P(\mathbb R^d)$.

In particular, it is obvious that the convergence in above holds for the whole sequence if the $T$-periodic solution is unique.
\end{proof}

\subsection{Continuous dependence}
In this subsection, we investigate the continuous dependence of periodic solutions for MVSDEs on parameters
by the continuous dependence of periodic Markov processes with state space $\mathbb R^d\times\mathcal P(\mathbb R^d)$.

\begin{thm}\label{main7}
Assume that the conditions of Theorem \ref{main3} hold with $b,\sigma$ replaced by $b_k,\sigma_k$,
and $\lim_{k\to\infty}f_k(t,x,\mu)=f(t,x,\mu)$ with $f=b,\sigma$.
Then $\{\mathcal L_{X_{k,t}}\}$ is tight,
where $X^k_t$ denote the $T$-periodic solutions for the following MVSDEs:
\begin{align*}
dX_{k,t}&= b_k(t, X_{k,t},\mathcal{L}_{X_{k,t}})dt+ \sigma_k(t, X_{k,t},\mathcal{L}_{X_{k,t}})dW_t.
\end{align*}
\end{thm}
\begin{proof}
By (H) and It\^o's formula, there exists a nonnegative constant $\lambda$ with $A_R>\lambda$ as $R$ is sufficiently large such that
$$LV(t,x,\mu)\le -A_R+\lambda,$$
and
\begin{align*}
EV(t,X_{k,t}, \mathcal L_{X_{k,t}})
& ={}
EV(s,X_{k,s}, \mathcal L_{X_{k,s}})+ E\int_s^tLV(r,X_{k,r}, \mathcal L_{X_{k,r}})dr\\
& \le{}
EV(s,X_{k,s}, \mathcal L_{X_{k,0}})- (A_R-\lambda) (t-s),
\end{align*}
where $X_{k}$ denotes the solution of the following MVSDE:
\begin{align}
\left\{
\begin{aligned}\label{pMVSDE1}
dX_{k,t}&= b_k(t,X_{k,t},\mathcal{L}_{X_{k,t}})dt+ \sigma_k(t,X_{k,t},\mathcal{L}_{X_{k,t}})dW_t,\\
X_{k,s}&= \xi.
\end{aligned}
\right.
\end{align}
By Gronwall's inequality, we get
$$EV(t,X_{k,t}, \mathcal L_{X_{k,t}})
\le
EV(s,X_{k,s}, \mathcal L_{X_{k,s}})e^{-(A_R-\lambda)(t-s)}
\le
EV(s,X_{k,s}, \mathcal L_{X_{k,s}})$$
for any $t\ge s$,
where the second inequality holds if $R$ is sufficiency large.
According to (H), there exists a constant $N>0$ such that
\begin{align*}
EV(t, X_{k,t}, \mathcal L_{X_{k,t}})
\ge{}&
P(V(t, X_{k,t},\mathcal L_{X_{k,t}})>R)\cdot R\\
={}&
P(|X_{k,t}|\vee |\mathcal L_{X_{k,t}}|_2>N)\cdot R\\
\ge {}&
P(|X_{k,t}|>N)\cdot R.
\end{align*}
Thus, we get that $\{\mathcal L_{X_{k,t}}\}$ is tight uniformly in $t$,
i.e. there exists a subsequence, still denoted $\{\mathcal L_{X_{k,t}}\}$, weakly converges to some probability $\nu_t$.

Particularly, by Theorem \ref{main3} there exists a $T$-periodic solution $X_{k,t}$ for any $k$.
Therefore, we obtain that $T$-periodic probability family $\{\mathcal L_{X_{k,t}}\}$ is tight.
\end{proof}
\begin{rem}
Denote $X_t$ by the $T$-periodic solution for the following MVSDE:
\begin{align*}
dX_{t}&= b(t, X_{t},\mathcal{L}_{X_{t}})dt+ \sigma(t, X_t,\mathcal{L}_{X_{t}})dW_t.
\end{align*}
If $\nu_t=\mathcal L_{X_t}$, then we obtain that periodic solutions continuously depends on parameter.
However, in order to guarantee the above equality to hold, we require some growth conditions to hold true; for details, see \cite{ML} for instance.
\end{rem}

\section{Applications}
In this section, we give some applications to illustrate our theoretical results.

\begin{exam}
Consider a auxiliary function
$$V(x,\mu):= |x|^2+ \int |y|^2\mu(dy).$$
We have
\begin{align*}
LV(x,\mu)
=&
2\langle x, b(t,x,\mu)\rangle+ |(\sigma\sigma^{\top})(t,x,\mu)|\\
&+
\int 2\langle y, b(t,y,\mu)\rangle+ |(\sigma\sigma^{\top})(t,y,\mu)|\mu(dy).
\end{align*}
Therefore, there exists a stationary ($T$-periodic) solution for MVSDE
\begin{align*}
dX_{t}&= b(t, X_{t},\mathcal{L}_{X_{t}})dt+ \sigma(t, X_t,\mathcal{L}_{X_{t}})dW_t.
\end{align*}
if the coefficients $b,\sigma$ are independent of $t$ ($T$-periodic in $t$),
and satisfy the following condition:
\begin{align*}
LV(x,\mu)
\to
-\infty \quad\hbox{ as} \quad |x|\vee|\mu|_2\to\infty.
\end{align*}
\end{exam}

\begin{exam}
Let
\begin{align*}
b(t,x,\mu)=-4x^3+\frac 18 xsint+ \int z\mu(dz),~
\sigma(t,x,\mu)=\sqrt 2x.
\end{align*}
Then we have
\begin{align*}
&2\langle x, b(t,x,\mu)\rangle+ |(\sigma\sigma^{\top})(t,x,\mu)|
+
\int 2\langle y, b(t,y,\mu)\rangle+ |(\sigma\sigma^{\top})(t,y,\mu)|\mu(dy)\\
= &
2\langle x,-4x^3+\frac 18 xsint+ \int z\mu(dz)\rangle+ 2|x|^2
+
\int 2\langle y,-4y^3+\frac 18 ysint+ \int z\mu(dz)\rangle+ 2|y|^2\mu(dy)\\
\le &
-8|x|^4+3\frac 14|x|^2
-\int 8|y|^4\mu(dy) +\int5\frac 14 |y|^2\mu(dy)\\
\to &
-\infty\quad\hbox{ as} \quad |x|\vee|\mu|_2\to\infty.
\end{align*}
Thus by Example 5.1, we obtain that the MVSDE with coefficients $b,\sigma$ has a $2\pi$-periodic solution.
\end{exam}

\begin{exam}
Let $X_t$ be a time-homogeneous stochastic process described by the MVSDE.
We consider another process,
which is different from $X_t$,
by the presence of a `force' $f(t,x,\mu)$ which is $T$-periodic in $t$:
\begin{align*}
dY_t=b(Y_t,\mathcal L_{Y_t})dt+\sigma(Y_t,\mathcal L_{Y_t})dW_t+f(t,Y_t,\mathcal L_{Y_t})dt.
\end{align*}
Suppose that the unperturbed system ($f\equiv 0$) has a stationary solution,
and there exists a Lyapunov function $V$ satisfying condition (H).
Then the perturbed system has a periodic solution if
$$\partial_xV(x,\mu)\cdot f(t,x,\mu)+\int\partial_{\mu}V(x,\mu)(y)\cdot f(t,y,\mu)\mu(dy)<C$$
for some constant $C$.
We directly prove that the Lyapunov function $V$ satisfies the assumptions of Theorem \ref{main8} for the stochastic process $Y_t$.

Specially, this result degenerates the classical situation for SDEs if the coefficients do not depend on distribution variable; see \cite[Example 3.8]{Khasminskii} for instance.
\end{exam}

\section*{Acknowledgement}

This work is supported by National Funded Postdoctoral Researcher Program (Grant
GZC20230414),
and Fundamental Research Funds for the Central Universities (Grant 135114002).

\end{document}